\documentclass[12pt,reqno]{amsart}
\usepackage[utf8]{inputenc}
\usepackage{amsthm,amsmath,amsxtra,latexsym,amscd,amssymb,exscale,mathrsfs,mathtools,xparse,xcolor}
\usepackage[top=3cm, bottom=3.5cm]{geometry}
\usepackage{geometry}
\usepackage{boxedminipage}
\usepackage[bookmarks=true]{hyperref}
\usepackage{bookmark}
\usepackage[abbrev]{amsrefs}

\newtheorem{theorem}{Theorem}[section]
\newtheorem{lemma}[theorem]{Lemma}

\newtheorem{corollary}[theorem]{Corollary}

\theoremstyle{definition}
 
\theoremstyle{remark}

\theoremstyle{remark}

\numberwithin{equation}{section}

\makeatletter
\def\subsection{\@startsection{subsection}{2}%
	\z@{.5\linespacing\@plus.7\linespacing}
	{.5\baselineskip}%
	{\normalfont\centering\scshape}%
}
\makeatother

\let\div\undefined
\DeclareMathOperator{\div}{div}

\let\vol\undefined
\DeclareMathOperator{\vol}{vol}

\DeclareMathOperator{\sff}{\mathrm{I\!\,I}}
\let\tr\undefined
\DeclareMathOperator{\tr}{tr}

\title[A logarithmic Sobolev inequality]{A logarithmic Sobolev inequality for closed submanifolds with constant length of mean curvature vector}
\author{Doanh Pham}
\address{Faculty of Mathematics and Applications, Saigon University, 273 An Duong Vuong Street, District 5, Ho Chi Minh City, Vietnam}
\email{ptdoanh@sgu.edu.vn}

\calclayout

\begin{document}
\maketitle
\begin{abstract}
In this paper, we prove a logarithmic Sobolev inequality for closed submanifolds with constant length of mean curvature vector in a manifold with nonnegative sectional curvature.
\end{abstract}
	
\section{Introduction}
In \cite{Brendle_cpam2022}, Brendle proved a sharp logarithmic Sobolev inequality for closed submanifolds in the Euclidean space. Later, following the strategies in \cite{Brendle_cpam2023}, several works \cite{Dong-Lin-Lu_ActaSci2024, Lee-Ricci_jga2024, Yi-Zheng_ChineseAnnals2024} extended the result in \cite{Brendle_cpam2022} to submanifolds of a manifold satisfying certain nonnegativity assumptions on curvatures. In \cite{Pham_imrn2024}, the author proved a logarithmic Sobolev inequality for closed minimal submanifolds of the unit sphere in the Euclidean space:

\begin{theorem}[\cite{Pham_imrn2024}]\label{Pham2024 logSobolev for all m}
Let $n, m \in \mathbb N$ and $\Sigma$ be a closed $n$-dimensional minimal submanifold of $\mathbb S^{n+m}$. Suppose that $f$ is a positive smooth function on $\Sigma$. If $m = 1, 2$, we have
\begin{equation}\label{eq:Pham2024 logSobolev for m=1,2}
\int_\Sigma f\left(\log f + \log(|\mathbb S^n|)\right) - \left(\int_\Sigma f \right) \log \left(\int_\Sigma f \right) \leq \frac{n+1}{2n^2} \int_\Sigma \frac{|\nabla^\Sigma f|^2}{f}.
\end{equation}
If $m \geq 3$, we have
\begin{equation}\label{eq:Pham2024 logSobolev for all m}
\int_\Sigma f\left(\log f + \log\left((n+1)\frac{|\mathbb S^{n+m}|}{|\mathbb S^{m-1}|}\right) \right) - \left(\int_\Sigma f \right) \log \left(\int_\Sigma f \right) \leq \frac{n+1}{2n^2} \int_\Sigma \frac{|\nabla^\Sigma f|^2}{f}.
\end{equation}
When $m=1, 2$, the equality in (\ref{eq:Pham2024 logSobolev for m=1,2}) holds if and only if $f$ is a constant and $\Sigma$ is totally geodesic.
\end{theorem}
Here, $\mathbb S^n$ denotes the unit sphere in the Euclidean space $\mathbb R^{n+1}$ which is the boundary of the unit ball $\mathbb B^{n+1} := \{x \in \mathbb R^{n+1}: |x| < 1\}$. In the cases when $m = 1, 2$, Theorem \ref{Pham2024 logSobolev for all m} may be compared with the classical result (see e.g. \cite{Beckner_Annals1993}) stating that
\begin{equation*}
\int_{\mathbb S^n} f\left(\log f + \log(|\mathbb S^n|) \right) - \left(\int_{\mathbb S^n}  f \right) \log \left(\int_{\mathbb S^n}  f \right) \leq \frac{1}{2n} \int_{\mathbb S^n}  \frac{|\nabla^{\mathbb S^n} f|^2}{f}
\end{equation*}
for every positive smooth function $f$ on $\mathbb{S}^n$ and the equality holds if and only if $f$ is a constant.

For a complete noncompact $k$-dimensional Riemannian manifold $M$ with nonnegative Ricci curvature, the asymptotic volume ratio is defined by
$$\theta \coloneqq \lim\limits_{r \to \infty} \frac{|B(p, r)|}{|\mathbb B^k| r^k},$$
where $p$ is an arbitrary point on $M$ and $B(p,r)$ denotes the geodesic ball of radius $r$ centered at $p$. To see that the limit in the definition of $\theta$ exists, one may recall Bishop-Gromov volume comparison theorem which shows that the function $r \mapsto \frac{|B(p, r)|}{|\mathbb B^k| r^k}$ is non-increasing and converges to $1$ as $ r \to 0^+$. Thus $\theta$ is well-defined and $0 \leq \theta \leq 1$. The main result of this paper is the following statement:

\begin{theorem}\label{logSobolev}
Let $n, m \in \mathbb N$, and $(M,g)$ be a complete noncompact Riemannian manifold of dimension $n+m$ with nonnegative sectional curvature and asymptotic volume ratio $\theta > 0$. Suppose that $\Sigma$ is a closed $n$-dimensional submanifold of $M$ such that the mean curvature vector $H$ of $\Sigma$ satisfies $|H| = 1$.  Furthermore, assume that $f$ is a positive smooth function on $\Sigma$. If $m =1, 2$, we have
\begin{equation}\label{eq:logSobolev m = 1}
	\int_\Sigma f\left(\log f + \log(|\mathbb S^n|) + \log \theta \right) - \left(\int_\Sigma f \right) \log \left(\int_\Sigma f \right) \leq \frac{n+1}{2n^2} \int_\Sigma \frac{|\nabla^\Sigma f|^2}{f}.
\end{equation}
If $m \geq 3$, we have
\begin{align}
\int_\Sigma f\left(\log f + \log\left((n+1)\frac{|\mathbb S^{n+m-1}|}{|\mathbb S^{m-2}|}\right) + \log \theta \right) - & \left(\int_\Sigma f \right) \log \left(\int_\Sigma f \right) \nonumber \\
& \qquad \leq \frac{n+1}{2n^2} \int_\Sigma \frac{|\nabla^\Sigma f|^2}{f}. \label{eq:logSobolev all m}
\end{align}
In either case, the equality holds if and only if $f$ is a constant, $\Sigma$ is totally umbilical, and
$$
|\Sigma| = 
\begin{cases}
\theta |\mathbb S^n| & \text{if } m = 1, 2,\\
\theta (n+1)\frac{|\mathbb S^{n+m-1}|}{|\mathbb S^{m-2}|} & \text{if } m \geq 3.
\end{cases}
$$

\end{theorem}

Despite being stated separately for exposition purposes, one may realize that (\ref{eq:logSobolev m = 1}) follows from (\ref{eq:logSobolev all m}). To see this, we note that when $m = 3$, the constant on the left-hand side of (\ref{eq:logSobolev all m}) reduces to $\log((n+1)|\mathbb S^{n+2}|/|\mathbb S^{1}|) = \log(|\mathbb S^n|)$. In the case $m = 1, 2$, the assertion follows by considering the larger ambient manifold $\widetilde{M} = M \times \mathbb R$ whose asymptotic volume ratio is the same as that of $M$.

When $M = \mathbb R^{n+m+1}$, Theorem \ref{logSobolev} is already a generalization of Theorem \ref{Pham2024 logSobolev for all m}. To see this, one just needs to note that the mean curvature vector with respect to $\mathbb R^{n+m+1}$ of a minimal submanifold of $\mathbb S^{n+m}$ is the opposite of the position vector. The combination of this observation, the estimates in \cite{Pham_imrn2024}, and the strategies in \cite{Brendle_cpam2023} is the key idea of the proof of the main result.

\section{Proof of the main result}
Let $n, m \in \mathbb N$, $m \geq 3$. From now till the end of this paper, we assume that $(M,g)$ is a complete noncompact Riemannian manifold of dimension $n+m$ with nonnegative sectional curvature and asymptotic volume ratio $\theta > 0$. The Levi-Civita connection on $M$ is denoted by $\overline{\nabla}$. Let $\Sigma$ be a closed $n$-dimensional submanifold of $M$ and $g_\Sigma$ be its induced metric. We denote by $\sff$ the the second fundamental form on $\Sigma$ and assume further that the mean curvature vector $H \coloneqq \tr(\sff)/n$ of $\Sigma$ satisfies $|H| = 1$. For each $x \in \Sigma$, we denote by $T_x^\perp\Sigma$ the space of normal vectors at $x$ and $\widetilde{T}_x^\perp\Sigma \coloneqq \{V \in T_x^\perp\Sigma : \langle V, H \rangle = 0\}$.
\subsection{Proof of inequality (\ref{eq:logSobolev all m})}
In this subsection, we prove inequality (\ref{eq:logSobolev all m}). We first suppose that $\Sigma$ is connected. Let $f$ be a positive smooth function on $\Sigma$. Since the inequality is invariant under multiplying $f$ by a positive constant, we may assume without loss of generality that $f$ satisfies
\begin{equation}\label{eq:scaling}
\frac{n}{n+1} \int_\Sigma f \log f = \frac{1}{2n} \int_\Sigma \frac{|\nabla^\Sigma f|^2}{f}.
\end{equation}
With this scaling, we need to show
\begin{equation}\label{eq:goal of proof}
	\theta (n+1)\frac{|\mathbb S^{n+m-1}|}{|\mathbb S^{m-2}|} \leq \int_\Sigma f.
\end{equation}
Since $f > 0$ and $\Sigma$ is connected, by (\ref{eq:scaling}), there exists (uniquely modulo a constant) a smooth solution $u$ to the equation
$$\div(f\,\nabla^\Sigma u) = \frac{n}{n+1} f \log f - \frac{1}{2n} \frac{|\nabla^\Sigma f|^2}{f} \quad \text{on }\; \Sigma.$$
We define some sets
\begin{align*}
\Omega & \coloneqq \{x \in \Sigma: |\nabla^\Sigma u(x)| < 1 \}\\
U & \coloneqq \{(x,y,t): x \in \Sigma,\, y \in \widetilde{T}_x^\perp\Sigma,\, t \in \mathbb{R} \text{ such that }  |\nabla^\Sigma u(x)|^2 + |y|^2 + t^2 < 1 \}.
\end{align*}
Furthermore, for each $r > 0$, we define
\begin{align*}
	A_r \coloneqq \Bigl\{(x, y, t)& \in U \text{ such that } \\
	&ru(z) + \frac{1}{2} d(z, \exp_x(r\nabla^\Sigma u(x) + ry + rtH))^2\\
	& \qquad \quad \geq ru(x) + \frac{1}{2} r^2(|\nabla^\Sigma u(x)|^2 + |y|^2 + t^2) \text{ for all } z \in \Sigma \Bigr\}.
\end{align*}
We also define the map $\Phi_r : \tilde{T}^\perp\Sigma \times \mathbb{R} \to M$ by
$$\Phi_r(x,y,t) = \exp_x(r \nabla^\Sigma u(x) + ry + rtH).$$
The proof of the following Lemma is taken from \cite[Lemma 2.4]{Pham_imrn2024}.

\begin{lemma}\label{Estimate of Delta u}
For every $x \in \Omega$, we have
$$\Delta_\Sigma u(x) \leq n\left(f(x)^{\frac{1}{n+1}} - \sqrt{1 - |\nabla^\Sigma u(x)|^2}\right).$$
\end{lemma}
\begin{proof}
For every point $x \in \Sigma$, the equation of $u$ implies
\begin{align*}
\Delta_\Sigma u(x) &= \frac{n}{n+1} \log f(x) - \frac{1}{2n}\frac{|\nabla^\Sigma f|^2}{f^2} - \left\langle  \frac{\nabla^\Sigma f}{f}, \nabla^\Sigma u \right\rangle.\\
& = \frac{n}{n+1} \log f(x) + \frac{n}{2} |\nabla^\Sigma u|^2 - \frac{1}{2}\left| \frac{1}{\sqrt{n}}\frac{\nabla^\Sigma f}{f} + \sqrt{n} \nabla^\Sigma u \right|^2\\
& \leq \frac{n}{n+1} \log f(x) + \frac{n}{2} |\nabla^\Sigma u|^2.
\end{align*}
Using the elementary inequalities
$$\log \lambda \leq \lambda - 1 \text{\; for } \lambda > 0 \quad \text{and} \quad \sqrt{1 - \theta} \leq 1 -\frac{\theta}{2} \text{\; for } 0 \leq \theta \leq 1,$$
and assuming $|\nabla^\Sigma u(x)| < 1$, we get
\begin{equation}\label{eq:logSobolev elem ineq}
	\frac{1}{n+1} \log f(x) \leq f(x)^{\frac{1}{n+1}} - 1 \quad \text{and} \quad \sqrt{1 - |\nabla^\Sigma u(x)|^2} \leq 1 -\frac{|\nabla^\Sigma u(x)|^2}{2}.
\end{equation}
Combining this and the last inequality in the chain of estimates for $\Delta u$, we deduce
\begin{equation}\label{eq:Laplacian u in logSobolev}
	\Delta_\Sigma u(x) \leq n\left( f(x)^{\frac{1}{n+1}} - 1 + \frac{|\nabla^\Sigma u|^2}{2}\right) \leq n\left(f(x)^{\frac{1}{n+1}} - \sqrt{1 - |\nabla^\Sigma u|^2}\right).
\end{equation}
The proof of the Lemma is complete.
\end{proof}

We state the following two Lemmas whose proofs are identical to those in \cite[Lemma 4.2 and Lemma 4.4]{Brendle_cpam2023}.

\begin{lemma}\label{likeBrendle_4.2}
For every $0 \leq \sigma < 1$ and $r > 0$, the set
$$\{p \in M: \sigma r < d(x,p) < r \text{ for all } x \in \Sigma\}$$
is a subset of the set
$$\{\Phi_r(x,y,t): (x,y,t) \in A_r  \text{ and } |\nabla^\Sigma u(x)|^2 + |y|^2 + t^2 > \sigma^2\}.$$
\end{lemma}

\begin{lemma}\label{likeBrendle_4.4}
For every $(x, y, t) \in A_r$, we have
$$g_\Sigma(x) + r (D_\Sigma^2u(x) - \langle\sff(x), y\rangle - t \langle\sff(x), H\rangle) \geq 0.$$
\end{lemma}

The second part of the following Lemma can also be seen from \cite[Proposition 4.6]{Brendle_cpam2023}. For the sake of exposition, we provide its proof since some details in the proof will be analyzed in the next subsection.

\begin{lemma}\label{likeBrendle_4.6}
Let $(x,y,t) \in A_r$. We have
\begin{equation}\label{eq:bound of t}
	1 + r \left(f(x)^{\frac{1}{n+1}} - \sqrt{1 - |\nabla^\Sigma u(x)|^2} - t \right) \geq 0.
\end{equation}
Moreover, the Jacobian determinant of $\Phi_s$ satisfies $\lim_{s \to 0^+} s^{-m}|\det \Phi_s(x,y,t)| = 1$, and the function
$$s \longmapsto \frac{|\det \Phi_s(x,y,t)|}{s^m \left(1 + s \left(f(x)^{\frac{1}{n+1}} - \sqrt{1 - |\nabla^\Sigma u(x)|^2} - t \right)\right)^n}$$
is non-increasing on $(0, r)$.
\end{lemma}

\begin{proof}
Let us fix a point $(x, y, t) \in A_r$ and denote $A = D_\Sigma^2u(x) - \langle\sff(x), y\rangle - t \langle\sff(x), H\rangle$. We first show (\ref{eq:bound of t}). By Lemma \ref{Estimate of Delta u}, we have
\begin{equation}\label{eq:trace of A}
\tr A = \Delta_\Sigma u(x) - nt \leq n\left(f(x)^{\frac{1}{n+1}} - \sqrt{1 - |\nabla^\Sigma u(x)|^2} - t \right)
\end{equation}
where we have used the fact that $\tr(\langle\sff(x), y\rangle) = 0$ and $\tr(\langle\sff(x), H\rangle) = n$. Since $ g_\Sigma(x) + rA \geq 0$ by Lemma \ref{likeBrendle_4.4}, we take its trace and apply (\ref{eq:trace of A}) to obtain
$$0 \leq n + r\tr A \leq n + nr\left(f(x)^{\frac{1}{n+1}} - \sqrt{1 - |\nabla^\Sigma u(x)|^2} - t \right).$$
This verifies (\ref{eq:bound of t}).

We now prove the second statement of the Lemma. We begin by choosing a local orthonormal frame $\{e_1, \dots, e_n, \nu_1, \dots, \nu_m\}$ in a neighborhood of $x$ in $T^\perp \Sigma$ in a way that $e_i \in T_{x}\Sigma$, $\nu_\alpha \in T^\perp_{x}\Sigma$, and $\langle \overline{\nabla}_{e_i} \nu_\alpha, \nu_\beta \rangle = 0$ at $x$. We let $\gamma(s) = \exp_{x}(s(\nabla^\Sigma u(x) + y + tH))$ for $s \in [0, r]$ and denote by $E_i(s)$, $N_\alpha(s)$ the parallel transports along $\gamma$ of $e_i$, $\nu_\alpha$ respectively. Furthermore, we let $X_i(t)$ and $Y_\alpha$ be the Jacobi fields along $\gamma$ satisfying
\begin{equation*}
\begin{cases*}
	X_i(0) = e_i, \\
	\langle D_t X_i(0), e_j \rangle = A(e_i, e_j)\\
	\langle D_t X_i(0), \nu_\beta \rangle = \langle \sff(e_i, \nabla^\Sigma u(x)), \nu_\beta \rangle
\end{cases*}
\qquad \text{and} \qquad
\begin{cases*}
	Y_\alpha(0) = 0,\\
	D_t Y_\alpha(0) = \nu_\alpha.
\end{cases*}
\end{equation*}
We also denote by $P(s)$ the square matrix-valued function on $[0,r]$ of size $(n+m)$ satisfying
\begin{align*}
P_{ij}(s) = \langle X_i(s), E_j(s) \rangle, \qquad P_{i\beta}(s) = \langle X_i(s), N_\beta(s) \rangle,\\
P_{\alpha j}(s) = \langle Y_\alpha(s),E_j(s) \rangle, \qquad P_{\alpha \beta}(s) = \langle Y_\alpha(s), N_\beta(s) \rangle.
\end{align*}
Following the argument in \cite[p. ~2208--2210]{Brendle_cpam2023}, we have that $\det P(s) > 0$ for every $s \in (0,r)$ and $\lim_{s \to 0^+} s^{-m} \det P(s) = 1$. Moreover, $|\det \Phi_s(x,y,t)| = \det P(s)$ for every $s \in(0,r)$. In addition, the matrix $Q(s) \coloneqq P(s)^{-1}P'(s)$ is symmetric for each $s \in (0,r)$ and its trace satisfies
\begin{equation}\label{eq:trace Q(s)}
\tr Q(s) \leq \frac{m}{s} + \sum\limits_{i = 1}^n \frac{\lambda_i}{1 + s\lambda_i}
\end{equation}
where $\lambda_1, \dots, \lambda_n$ are the eigenvalues of the matrix $A$. Applying the arithmetic-harmonic mean inequality, we find that
$$\sum\limits_{i = 1}^n \frac{1}{1 + s\lambda_i} \geq \frac{n^2}{\sum_{i=1}^{n}(1 + s\lambda_i)} = \frac{n}{1 + \frac{s}{n}\tr A}.$$
Hence, for $s \in (0,r)$, we have
\begin{equation}\label{eq:trace Q(s) 2}
\tr Q(s) \leq \frac{m}{s} + \sum\limits_{i = 1}^n \frac{\lambda_i}{1 + s\lambda_i} = \frac{m}{s} + \frac{1}{s}\left(n - \sum\limits_{i = 1}^n \frac{1}{1 + s\lambda_i} \right) \leq \frac{m}{s} + \frac{\tr A}{1 + \frac{s}{n}\tr A}.
\end{equation}
Combining (\ref{eq:trace of A}) and (\ref{eq:trace Q(s) 2}), we infer that
$$\tr Q(s) \leq \frac{m}{s} + \frac{n\left(f(x)^{\frac{1}{n+1}} - \sqrt{1 - |\nabla^\Sigma u(x)|^2} - t \right)}{1 + s\left(f(x)^{\frac{1}{n+1}} - \sqrt{1 - |\nabla^\Sigma u(x)|^2} - t \right)} \quad \text{for } s \in (0,r).$$
Since $\frac{d}{ds} \log \det P(s) = \tr Q(s)$, it follows from the previous inequality that
$$\frac{d}{ds} \frac{\det P(s)}{s^m \left(1 + s \left(f(x)^{\frac{1}{n+1}} - \sqrt{1 - |\nabla^\Sigma u(x)|^2} - t \right)\right)^n} \leq 0 \quad \text{on } (0,r).$$
The proof of the Lemma is complete.
\end{proof}

As an immediate consequence, we obtain an upper bound for the Jacobian determinant of $\Phi_r$ on $A_r$.

\begin{corollary}\label{Estimate for det(Phi)}
For every $(x,y,t) \in A_r$, we have
\begin{equation}\label{eq:Estimate for det(Phi)}
|\det D\Phi_r(x,y,t)| \leq r^{m}\left(1 + r \left(f(x)^{\frac{1}{n+1}} - \sqrt{1 - |\nabla^\Sigma u(x)|^2} - t \right)\right)^n.
\end{equation}
\end{corollary}

From (\ref{eq:bound of t}), we note that for $(x,y,t) \in A_r$, the range of $t$ is given by
\begin{equation}\label{eq:range of t}
-\sqrt{1 - |\nabla^\Sigma u(x)|^2} < t \leq f(x)^{\frac{1}{n+1}} - \sqrt{1 - |\nabla^\Sigma u(x)|^2} + \frac{1}{r}.
\end{equation}
To continue with the proof of (\ref{eq:logSobolev all m}), we apply Lemma \ref{likeBrendle_4.2}, Corollary \ref{Estimate for det(Phi)}, and the area formula to get
\begin{align*}
& |\{p \in M: \sigma r < d(x,p) < r \text{ for all } x \in \Sigma\}|\\
& \leq \int_\Omega \int_{-\sqrt{1 - |\nabla^\Sigma u|^2}}^{f(x)^{\frac{1}{n+1}} - \sqrt{1 - |\nabla^\Sigma u|^2} + \frac{1}{r}} \int_{\{y \in \widetilde{T}_x^\perp \Sigma:\; \sigma^2 < |\nabla^\Sigma u|^2 + |y|^2 + t^2 < 1\}}\\
&\qquad \qquad \qquad \qquad \qquad \qquad \qquad |\det D\Phi_r(x,y,t)|\, 1_{A_r}(x,y,t)\, dy\, dt\, d\vol_\Sigma(x)\\
& \leq \int_\Omega \int_{-\sqrt{1 - |\nabla^\Sigma u|^2}}^{f(x)^{\frac{1}{n+1}} - \sqrt{1 - |\nabla^\Sigma u|^2} + \frac{1}{r}} \int_{\{y \in \widetilde{T}_x^\perp \Sigma:\; \sigma^2 < |\nabla^\Sigma u|^2 + |y|^2 + t^2 < 1\}}\\
&\qquad \qquad \qquad r^{m}\left(1 + r \left(f(x)^{\frac{1}{n+1}} - \sqrt{1 - |\nabla^\Sigma u(x)|^2} - t \right)\right)^n \, dy\, dt\, d\vol_\Sigma(x).
\end{align*}
Moreover, by the elementary inequality $b^\frac{m-1}{2} - a^\frac{m-1}{2} \leq \frac{m-1}{2}(b-a)$ for $0 \leq a \leq b < 1$ and $m \geq 3$, for every $x \in \Omega$, we have
\begin{align*}
|\{y \in &T_x^\perp \Sigma : \sigma^2 \leq |\nabla^\Sigma u|^2 + |y|^2 + t^2 < 1\}|\\
& \leq |\mathbb B^{m-1}| \left((1 - |\nabla^\Sigma u|^2 - t^2)^\frac{m-1}{2}_+ - (\sigma^2 - |\nabla^\Sigma u|^2 - t^2)^\frac{m-1}{2}_+ \right)\\
& \leq \frac{m-1}{2} |\mathbb B^{m-1}|(1-\sigma^2).
\end{align*}
Therefore, we continue the chain of integral estimates by
\begin{align*}
|\{p & \in M:\sigma r < d(x,p) < r \text{ for all } x \in \Sigma\}|\\
& \leq \frac{m-1}{2} |\mathbb B^{m-1}|(1-\sigma^2)\int_\Omega \int_{-\sqrt{1 - |\nabla^\Sigma u|^2}}^{f(x)^{\frac{1}{n+1}} - \sqrt{1 - |\nabla^\Sigma u|^2} + \frac{1}{r}}\\
& \qquad \qquad \qquad r^{m}\left(1 + r \left(f(x)^{\frac{1}{n+1}} - \sqrt{1 - |\nabla^\Sigma u(x)|^2} - t \right)\right)^n \, dt\, d\vol_\Sigma(x)\\
& = \frac{m-1}{2} |\mathbb B^{m-1}|(1-\sigma^2) \int_\Omega \frac{r^{m-1}}{n+1} \left(1 + rf(x)^{\frac{1}{n+1}}\right)^{n+1} d\vol_\Sigma(x).
\end{align*}
Dividing the above inequality by $r^{n+m}$ and passing $r \to \infty$, we conclude that
$$\theta |\mathbb B^{n+m}| (1 - \sigma^{n+m}) \leq \frac{m-1}{2(n+1)} |\mathbb B^{m-1}|(1-\sigma^2) \int_\Omega f(x) \, d\vol_\Sigma(x).$$
Finally, we divide the previous inequality by $1 - \sigma$ and let $\sigma \to 1$ to obtain
\begin{equation}\label{eq:logSobolev finished}
\theta(n+m) |\mathbb B^{n+m}| \leq \frac{(m-1)|\mathbb B^{m-1}|}{n+1}\int_\Omega f \leq \frac{(m-1)|\mathbb B^{m-1}|}{n+1}\int_\Sigma f, 
\end{equation}
which coincides with (\ref{eq:goal of proof}). This finishes the proof of (\ref{eq:logSobolev all m}) when $\Sigma$ is connected.

Now, we suppose that $\Sigma$ is disconnected. Since (\ref{eq:logSobolev all m}) holds on each individual of connected component of $\Sigma$, we take the sum over them and use the elementary inequality
$$a \log a + b\log b < (a+b)\log(a+b) \quad \text{for all } a, b > 0$$
to finish the proof in this case.

\subsection{On the equality cases}
In this subsection, we analyze the equality case of (\ref{eq:logSobolev all m}). To begin with, we assume that $f$ is a positive smooth function which satisfies
\begin{align}
\int_\Sigma f\left(\log f + \log\left((n+1)\frac{|\mathbb S^{n+m-1}|}{|\mathbb S^{m-2}|}\right) + \log \theta \right) - & \left(\int_\Sigma f \right) \log \left(\int_\Sigma f \right) \nonumber \\
& \qquad = \frac{n+1}{2n^2} \int_\Sigma \frac{|\nabla^\Sigma f|^2}{f}. \label{eq:logSobolev rigidity}
\end{align}
From the last paragraph in the proof of (\ref{eq:logSobolev all m}), we infer that $\Sigma$ is connected. Since (\ref{eq:logSobolev rigidity}) is invariant under multiplication of $f$ by a positive constant, we assume further without loss of generality that $f$ satisfies
\begin{equation}\label{eq:scaling in logSobolev rigidity}
	\frac{n}{n+1} \int_\Sigma f \log f = \frac{1}{2n} \int_\Sigma \frac{|\nabla^\Sigma f|^2}{f}.
\end{equation}
By (\ref{eq:logSobolev rigidity}), the scaling of $f$ implies
\begin{equation}\label{eq:logSobolev rigidity goal}
	\theta (n+1)\frac{|\mathbb S^{n+m-1}|}{|\mathbb S^{m-2}|} = \int_\Sigma f.
\end{equation}
Since $f > 0$ and $\Sigma$ is connected, by (\ref{eq:scaling in logSobolev rigidity}), there exists (uniquely modulo a constant) a smooth solution $u$ to the equation
$$\div(f\,\nabla^\Sigma u) = \frac{n}{n+1} f \log f - \frac{1}{2n} \frac{|\nabla^\Sigma f|^2}{f} \quad \text{on }\, \Sigma.$$
We define the sets $\Omega$,  $U$, $A_r$, and the map $\Phi_r$ as in the proof of (\ref{eq:logSobolev all m}).
It follows from (\ref{eq:logSobolev finished}) and (\ref{eq:logSobolev rigidity goal}) that $\Omega$ is dense in $\Sigma$.

\begin{lemma}\label{det DPhi in equality case}
For every $r > 0$, $x \in \Omega$, $y \in \widetilde{T}_x^\perp\Sigma$ and $t \in [-1, 1]$ satisfying $|\nabla^\Sigma u(x)|^2 + |y|^2 + t^2 = 1$, we have
$$|\det D\Phi_r(x,y,t)| \geq r^{m}\left(1 + r \left(f(x)^{\frac{1}{n+1}} - \sqrt{1 - |\nabla^\Sigma u(x)|^2} - t \right)\right)^n.$$
\end{lemma}

\begin{proof}
We prove the Lemma by a contradiction argument. To do this, we suppose on the contrary that the statement does not hold at $x_0 \in \Omega$, $y_0 \in \widetilde{T}_{x_0}^\perp\Sigma$, $t_0 \in [-1, 1]$ satisfying $|\nabla^\Sigma u(x_0)|^2 + |y_0|^2 + t_0^2 = 1$ for some $r_0 > 0$. Then, by continuity, there exist $0 < \varepsilon < 1$ and a neighborhood $V$ of $(x_0, y_0, t_0)$ with the property that
$$|\det D\Phi_{r_0}(x,y,t)| < (1 - \varepsilon)r_0^{m}\left(1 + r_0 \left(f(x)^{\frac{1}{n+1}} - \sqrt{1 - |\nabla^\Sigma u(x)|^2} - t \right)\right)^n \text{ on } V.$$
By Lemma \ref{likeBrendle_4.6}, it follows that for each $r > r_0$,
$$|\det D\Phi_{r}(x,y,t)| < (1 - \varepsilon)r^{m}\left(1 + r \left(f(x)^{\frac{1}{n+1}} - \sqrt{1 - |\nabla^\Sigma u(x)|^2} - t \right)\right)^n \text{ on } V\cap A_r.$$
From this, for $r > r_0$ and $0 < \sigma < 1$, we apply Lemma \ref{likeBrendle_4.2}, the area formula and (\ref{eq:range of t}) to infer that
\begin{align*}
	& |\{p \in M: \sigma r < d(x,p) < r \text{ for all } x \in \Sigma\}|\\
	& \leq \int_\Omega \int_{-\sqrt{1 - |\nabla^\Sigma u|^2}}^{f(x)^{\frac{1}{n+1}} - \sqrt{1 - |\nabla^\Sigma u|^2} + \frac{1}{r}} \int_{\{y \in \widetilde{T}_x^\perp \Sigma:\; \sigma^2 < |\nabla^\Sigma u|^2 + |y|^2 + t^2 < 1\}}\\
	&\qquad \qquad \qquad \qquad \qquad \qquad \qquad \qquad \quad |\det D\Phi_r(x,y,t)|\, 1_{A_r}(x,y,t)\, dy\, dt\, d\vol_\Sigma(x)\\
	& \leq \int_\Omega \int_{-\sqrt{1 - |\nabla^\Sigma u|^2}}^{f(x)^{\frac{1}{n+1}} - \sqrt{1 - |\nabla^\Sigma u|^2} + \frac{1}{r}} \int_{\{y \in \widetilde{T}_x^\perp \Sigma:\; \sigma^2 < |\nabla^\Sigma u|^2 + |y|^2 + t^2 < 1\}}\\
	&\qquad (1 - \varepsilon 1_V(x,y,t)) r^{m}\left(1 + r \left(f(x)^{\frac{1}{n+1}} - \sqrt{1 - |\nabla^\Sigma u(x)|^2} - t \right)\right)^n \, dy\, dt\, d\vol_\Sigma(x).
\end{align*}
Similarly to the proof of Theorem \ref{logSobolev}, we continue the chain of estimates by
\begin{align*}
& |\{p \in M: \sigma r < d(x,p) < r \text{ for all } x \in \Sigma\}|\\
& \leq \frac{m-1}{2} |\mathbb B^{m-1}|(1-\sigma^2) \int_\Omega \frac{r^{m-1}}{n+1} \left(1 + rf(x)^{\frac{1}{n+1}}\right)^{n+1} d\vol_\Sigma(x)\\
& \qquad - \varepsilon \int_\Omega \int_{-\sqrt{1 - |\nabla^\Sigma u|^2}}^{f(x)^{\frac{1}{n+1}} - \sqrt{1 - |\nabla^\Sigma u|^2} + \frac{1}{r}} \int_{\{y \in \widetilde{T}_x^\perp \Sigma:\; \sigma^2 < |\nabla^\Sigma u|^2 + |y|^2 + t^2 < 1\}}\\
&\qquad \qquad \quad 1_V(x,y,t) r^{m}\left(1 + r \left(f(x)^{\frac{1}{n+1}} - \sqrt{1 - |\nabla^\Sigma u(x)|^2} - t \right)\right)^n \, dy\, dt\, d\vol_\Sigma(x).
\end{align*}
Dividing the above inequality by $r^{n+m}$ and passing to the limit as $r \to \infty$, we obtain
\begin{equation}\label{eq:divide by r^{n+m} equality case}
\theta |\mathbb B^{n+m}|(1 - \sigma^{n+m}) \leq \frac{m-1}{2(n+1)} |\mathbb B^{m-1}|(1-\sigma^2) \int_\Omega f(x) \, d\vol_\Sigma(x) - \varepsilon \lim\limits_{r \to \infty} I(\sigma, r)
\end{equation}
where
\begin{align*}
I(\sigma, r) \coloneqq &\int_\Omega \int_{-\sqrt{1 - |\nabla^\Sigma u|^2}}^{f(x)^{\frac{1}{n+1}} - \sqrt{1 - |\nabla^\Sigma u|^2} + \frac{1}{r}} \int_{\{y \in \widetilde{T}_x^\perp \Sigma:\; \sigma^2 < |\nabla^\Sigma u|^2 + |y|^2 + t^2 < 1\}}\\
&\qquad \qquad \quad 1_V(x,y,t) \left(f(x)^{\frac{1}{n+1}} - \sqrt{1 - |\nabla^\Sigma u(x)|^2} - t \right)^n \, dy\, dt\, d\vol_\Sigma(x)\\
& \geq \int_\Omega \int_{-\sqrt{1 - |\nabla^\Sigma u|^2}}^{f(x)^{\frac{1}{n+1}} - \sqrt{1 - |\nabla^\Sigma u|^2}} \int_{\{y \in \widetilde{T}_x^\perp \Sigma:\; \sigma^2 < |\nabla^\Sigma u|^2 + |y|^2 + t^2 < 1\}}\\
&\qquad \qquad \quad 1_V(x,y,t) \left(f(x)^{\frac{1}{n+1}} - \sqrt{1 - |\nabla^\Sigma u(x)|^2} - t \right)^n \, dy\, dt\, d\vol_\Sigma(x)\\
& \eqqcolon J(\sigma).
\end{align*}
Noting that $\lim_{\sigma \to 1} \frac{J(\sigma)}{1-\sigma} > 0$, we divide (\ref{eq:divide by r^{n+m} equality case}) by $1 - \sigma$ and pass to the limit as $\sigma \to 1$ to conclude that
\begin{equation}\label{eq:logSobolev equality case det(DPhi) finished}
\theta(n+m) |\mathbb B^{n+m}| < \frac{(m-1)|\mathbb B^{m-1}|}{n+1}\int_\Omega f \leq \frac{(m-1)|\mathbb B^{m-1}|}{n+1}\int_\Sigma f
\end{equation}
which contradicts (\ref{eq:logSobolev rigidity goal}). This finishes the proof of the Lemma.
\end{proof}

To continue, we fix a triplet $(x,y,t)$ where $x \in \Omega$, $y \in \widetilde{T}_x^\perp\Sigma$ and $t \in [-1, 1]$ satisfying $|\nabla^\Sigma u(x)|^2 + |y|^2 + t^2 = 1$. We define the matrix $A = D_\Sigma^2u(x) - \langle\sff(x), y\rangle - t \langle\sff(x), H\rangle$. There exists a real number $s_0 > 0$ such that the matrix $g_\Sigma + s A$ is positive definite for every $s \in (0, s_0)$. Since $x \in \Omega$, by Lemma \ref{Estimate of Delta u}, we have
\begin{equation}\label{eq:trace of A equality case}
\tr A = \Delta_\Sigma u(x) - nt \leq n\left(f(x)^{\frac{1}{n+1}} - \sqrt{1 - |\nabla^\Sigma u(x)|^2} - t \right).
\end{equation}

Next, we choose a local orthonormal frame $\{e_1, \dots, e_n, \nu_1, \dots, \nu_m\}$ in a neighborhood of $x$ in $T^\perp \Sigma$ in a way that $e_i \in T_{x}\Sigma$, $\nu_\alpha \in T^\perp_{x}\Sigma$, and $\langle \overline{\nabla}_{e_i} \nu_\alpha, \nu_\beta \rangle = 0$ at $x$. We let $\gamma(s) = \exp_{x}(s(\nabla^\Sigma u(x) + y + tH))$ for $s \in (0, s_0)$ and denote by $E_i(s)$, $N_\alpha(s)$ the parallel transports along $\gamma$ of $e_i$, $\nu_\alpha$ respectively. Furthermore, we let $X_i(t)$ and $Y_\alpha$ be the Jacobi fields along $\gamma$ satisfying
\begin{equation*}
	\begin{cases*}
		X_i(0) = e_i, \\
		\langle D_t X_i(0), e_j \rangle = A(e_i, e_j)\\
		\langle D_t X_i(0), \nu_\beta \rangle = \langle \sff(e_i, \nabla^\Sigma u(x)), \nu_\beta \rangle
	\end{cases*}
	\qquad \text{and} \qquad
	\begin{cases*}
		Y_\alpha(0) = 0,\\
		D_t Y_\alpha(0) = \nu_\alpha.
	\end{cases*}
\end{equation*}
We also denote by $P(s)$ the square matrix-valued function on $[0, s_0)$ of size $(n+m)$ satisfying
\begin{align*}
	P_{ij}(s) = \langle X_i(s), E_j(s) \rangle, \qquad P_{i\beta}(s) = \langle X_i(s), N_\beta(s) \rangle,\\
	P_{\alpha j}(s) = \langle Y_\alpha(s),E_j(s) \rangle, \qquad P_{\alpha \beta}(s) = \langle Y_\alpha(s), N_\beta(s) \rangle.
\end{align*}
Then, by Lemma \ref{det DPhi in equality case}, we have
\begin{equation*}
|\det P(s)| = |\det D\Phi_s(x, y, t)| \geq  s^{m}\left(1 + s \left(f(x)^{\frac{1}{n+1}} - \sqrt{1 - |\nabla^\Sigma u(x)|^2} - t \right)\right)^n
\end{equation*}
for all $s \in (0, s_0)$. Since $\det P(s) > 0$ for sufficiently small $s >0$, we find that
\begin{equation}\label{eq:det P(s) in equality case}
\det P(s) \geq  s^{m}\left(1 + s \left(f(x)^{\frac{1}{n+1}} - \sqrt{1 - |\nabla^\Sigma u(x)|^2} - t \right)\right)^n \quad \text{for } s \in (0, s_0).
\end{equation}
Similarly to the proof of Lemma \ref{likeBrendle_4.6}, the matrix $Q(s) \coloneqq P(s)^{-1}P'(s)$ is symmetric for each $s \in (0,s_0)$ and its trace satisfies
\begin{equation}\label{eq:trace Q(s) equality case}
	\tr Q(s) \leq \frac{m}{s} + \sum\limits_{i = 1}^n \frac{\lambda_i}{1 + s\lambda_i}
\end{equation}
where $\lambda_1, \dots, \lambda_n$ are the eigenvalues of the matrix $A$. Applying the arithmetic-harmonic mean inequality, we find that
\begin{equation}\label{eq:HM-AM equality case}
\sum\limits_{i = 1}^n \frac{1}{1 + s\lambda_i} \geq \frac{n^2}{\sum_{i=1}^{n}(1 + s\lambda_i)} = \frac{n}{1 + \frac{s}{n}\tr(A)}.
\end{equation}
Hence, for $s \in (0,s_0)$, we have
\begin{equation}\label{eq:trace Q(s) 2 equality case}
\tr Q(s) \leq \frac{m}{s} + \sum\limits_{i = 1}^n \frac{\lambda_i}{1 + s\lambda_i} = \frac{m}{s} + \frac{1}{s}\left(n - \sum\limits_{i = 1}^n \frac{1}{1 + s\lambda_i} \right) \leq \frac{m}{s} + \frac{\tr A}{1 + \frac{s}{n}\tr A}.
\end{equation}
Combining (\ref{eq:trace of A equality case}) and (\ref{eq:trace Q(s) 2 equality case}), we infer that
$$\tr Q(s) \leq \frac{m}{s} + \frac{n\left(f(x)^{\frac{1}{n+1}} - \sqrt{1 - |\nabla^\Sigma u(x)|^2} - t \right)}{1 + s\left(f(x)^{\frac{1}{n+1}} - \sqrt{1 - |\nabla^\Sigma u(x)|^2} - t \right)} \quad \text{for } s \in (0,s_0).$$
Since $\frac{d}{ds} \log \det P(s) = \tr Q(s)$, it follows that
\begin{equation*}
\det P(s) \leq s^{m}\left(1 + s \left(f(x)^{\frac{1}{n+1}} - \sqrt{1 - |\nabla^\Sigma u(x)|^2} - t \right)\right)^n \quad \text{for } s \in (0, s_0).
\end{equation*}
Recalling (\ref{eq:det P(s) in equality case}), we deduce that
\begin{equation*}
\det P(s) = s^{m}\left(1 + s \left(f(x)^{\frac{1}{n+1}} - \sqrt{1 - |\nabla^\Sigma u(x)|^2} - t \right)\right)^n \quad \text{for } s \in (0, s_0).
\end{equation*}
Therefore, we have
\begin{equation}\label{eq:trace Q(s) equality}
\tr Q(s) = \frac{m}{s} + \frac{n\left(f(x)^{\frac{1}{n+1}} - \sqrt{1 - |\nabla^\Sigma u(x)|^2} - t \right)}{1 + s\left(f(x)^{\frac{1}{n+1}} - \sqrt{1 - |\nabla^\Sigma u(x)|^2} - t \right)} \quad \text{for } s \in (0,s_0).
\end{equation}
This implies that the equalities must hold in (\ref{eq:trace of A equality case}) and (\ref{eq:trace Q(s) 2 equality case}). Moreover, in view of (\ref{eq:logSobolev elem ineq}), the equality in (\ref{eq:trace of A equality case}) gives $f(x) = 1$ and $\nabla^\Sigma u(x) = 0$. Since $x \in \Omega$ is arbitrarily chosen and $\Omega$ is dense in $\Sigma$, we conclude that $f \equiv 1$ and $u$ is a constant on $\Sigma$. On the other hand, the equalities in (\ref{eq:trace Q(s) 2 equality case}) imply that the equality in (\ref{eq:HM-AM equality case}) holds. This means that all the eigenvalues $\lambda_i$ of the matrix $A$ have the same value. Since $u$ is a constant on $\Sigma$ and $(y, t)$ is arbitrarily chosen only to satisfy $|y|^2 + t^2 = 1$, we conclude that $\Sigma$ is totally umbilical. The proof of the main result is complete.

\end{document}